\documentclass{article}
\usepackage{amssymb}
\usepackage{amssymb}
\usepackage{amsmath}
\usepackage{amsfonts}
\usepackage{graphicx}
\usepackage{soul, color}
\usepackage{tikz}
\usepackage{hyperref}
\newtheorem{theorem}{Theorem}

\newtheorem{corollary}[theorem]{Corollary}

\newtheorem{lemma}[theorem]{Lemma}

\newtheorem{proposition}[theorem]{Proposition}
\newtheorem{remark}[theorem]{Remark}

\newenvironment{proof}[1][Proof]{\noindent\textbf{#1.} }{\ \rule{0.5em}{0.5em}}
\input{tcilatex}

\newcommand{\ud}{\,\mathrm{d}}
\newcommand{\p}{\ensuremath{\partial}}
\newcommand{\n}{\ensuremath{\nonumber}}
\newcommand{\eps}{\ensuremath{\varepsilon}}

\newcommand{\E}{\ensuremath{\mathcal{E}}}

\newcommand{\bigO}{\mathcal{O}}

 \allowdisplaybreaks

\title{\vspace{-50pt} Stationary Inviscid Limit to Shear Flows}
\author{ \Large Sameer Iyer \footnote{\url{sameer_iyer@brown.edu}. Division of Applied Mathematics, Brown University, 182 George Street, Providence, RI 02912, USA. Partially supported by NSF grant DMS-1611695} \hspace{10 mm} Chunhui Zhou \footnote{\url{zhouchunhui@seu.edu.cn}. School of Mathematics, Southeast University, Nanjing, China.} }
\date{November 15, 2017}

\begin{document}

\maketitle

\begin{abstract}
In this note we establish a density result for certain stationary shear flows, $\mu(y)$, that vanish at the boundaries of a horizontal channel. We construct stationary solutions to 2D Navier-Stokes that are $\eps$-close in $L^\infty$ to the given shear flow. Our construction is based on a coercivity estimate for the Rayleigh operator, $R[v]$, which is based on a decomposition made possible by the vanishing  of $\mu$ at the boundaries. 
\end{abstract}

\section{Introduction}

We are considering 2D, stationary flows on the strip: 
\begin{align}
\Omega = (0,L) \times (0,2). 
\end{align}

We consider an Euler shear flow:
\begin{align}
\bold{u}^0 = (\mu(y), 0). 
\end{align}

Let $\bold{u}^\eps$ solve the Navier-Stokes equations: 
\begin{align} \label{main.NS}
\left.
\begin{aligned}
&\bold{u}^\eps \cdot \nabla \bold{u}^\eps + \nabla P^\eps = \eps \Delta \bold{u}^\eps \\
&\nabla \cdot \bold{u}^\eps = 0 \\
&\bold{u}^\eps|_{y = 0} = 0,  \bold{u}^\eps|_{y = 2} = u_b 
\end{aligned}
\right\}.
\end{align}

Here $u_b \ge 0$ denotes the velocity of the boundary at $\{y = 2\}$. Our main result, Theorem \ref{Th.no.slip} treats the non-moving case of $u_b = 0$. We are interested in the asymptotic behavior of $\bold{u}^\eps$ as $\eps \rightarrow 0$. In the presence of boundaries, the vanishing viscosity asymptotics are a major open problem in fluids made challenging due to the mismatch between the no-slip condition $\bold{u}^\eps|_{\p \Omega} = 0$ and the no penetration condition typically satisfied by Euler flows: $\bold{u}^0 \cdot \bold{n} = 0$. This mismatch is typically rectified by the presence of Prandtl's boundary layer (see \cite{GN}, \cite{Iyer}, \cite{Iyer2}, \cite{Iyer3} for relevant results in the 2D stationary setting). In this article, we will consider Euler flows that themselves satisfy no-slip: 
\begin{align}
\mu(0) = 0, 
\end{align}

for which there is no leading order boundary layer. Denote now the asymptotic expansion: 
\begin{align} \label{asy.ex.1}
\bold{u}^\eps := \begin{pmatrix} u^\eps \\  v^\eps \end{pmatrix} =\begin{pmatrix} \mu + \eps u^1_e + \eps u^1_p + \eps^{\frac{3}{2}} u^2_e + \eps^{\frac{3}{2}} u^2_p + \eps^{\frac{3}{2}+\gamma}u \\  \eps v^1_e + \eps^{\frac{3}{2}} v^1_p + \eps^{\frac{3}{2}} v^2_e + \eps^{2} v^2_p + \eps^{\frac{3}{2}+\gamma}v \end{pmatrix}. 
\end{align}

We denote: 
\begin{align}
&u_s := \mu + \eps u^1_e + \eps u^1_p + \eps^{\frac{3}{2}} u^2_e + \eps^{\frac{3}{2}} u^2_p, \\
&v_s := \eps v^1_e + \eps^{\frac{3}{2}} v^1_p + \eps^{\frac{3}{2}} v^2_e + \eps^{2} v^2_p.
\end{align}

We impose the boundary conditions: 
\begin{align} \label{bc.d}
&[u,v]|_{x = 0} = [u,v]|_{y = 0} = [u,v]|_{y = 2} = 0, \\ \label{bc.sf}
& \p_y u + \p_x v = 0, \hspace{3 mm} P = 2\eps \p_x u. 
\end{align}

The system satisfied by $[u,v]$ is: 
\begin{align} \label{sys.1}
\left.
\begin{aligned}
& -\eps \Delta u +S_u + \p_x P = f:=  \mathcal{N}_1(u,v) + \mathcal{F}_u \\
& -\eps \Delta v + S_v + \p_y P = g:= \mathcal{N}_2(u,v) + \mathcal{F}_v \\
& \p_x u + \p_y v = 0
\end{aligned}
\right\} \text{ in } \Omega. 
\end{align}

We have defined: 
\begin{align}
&S_u := u_s u_x + u_{sx}u + u_{sy} v + v_s u_y, \hspace{2 mm} S_v := u_s v_x + v_{sx}u + v_s v_y + v v_{sy}, \\ \label{defn.N}
&\mathcal{N}_1 := -\eps^{\frac{3}{2}+\gamma} \Big( u\p_x u + v \p_y u \Big), \hspace{3 mm} \mathcal{N}_2 := -\eps^{\frac{3}{2}+\gamma} \Big( u\p_x v + v \p_y v \Big), 
\end{align}

and $\mathcal{F}_{u}, \mathcal{F}_v$ are defined in (\ref{Fu.Fv}). Let us now define several norms in which we will control the solution: 
\begin{align}
&||u,v||_{\mathcal{E}} := ||\sqrt{\eps} \nabla u||_{L^2} + ||\sqrt{\eps} \nabla v||_{L^2}, \\
&||u,v||_{\mathcal{P}} := ||\sqrt{u_s} \nabla v||_{L^2}, \\
&||u,v||_{X} := ||u,v||_{\mathcal{E}} + \eps^{\frac{\gamma}{2}} ||\sqrt{\eps} \{ u, v \}||_{\infty}
\end{align}

We introduce here the notation: 
\begin{align}
\tilde{y} = y \cdot (2-y).
\end{align}

The main theorems we prove are the following: 
\begin{theorem} \label{Th.no.slip} Let $u_b \ge 0$ in (\ref{main.NS}). Let $\mu(y) \in C^\infty([0,2])$ be a given function, satisfying the conditions: 
\begin{align} \label{as.1}
&\mu(0) = 0, \mu(2) = u_b, \\ \label{as.2}
&\p_y^j \mu(0) = \p_y^j \mu(2) = 0 \text{ for } 2 \le j \le N_0, \\
&\p_y \mu(0) > 0, |\p_y \mu(2)| > 0.
\end{align}

where $N_0 < \infty$ and large but unspecified.\footnote{We have selected not to optimize $N_0$. The optimal $N_0$ is likely between 4 and 10.} Let also standard compatibility conditions at the corners of $\Omega$ be prescribed for the layers in $u_s$.\footnote{ We omit stating the precise form of these compatibility conditions here. They can be found in (\ref{compat.1}), (\ref{comp.BL}).} Then there exists a unique solution, $\bold{u}^\eps$ satisfying the Navier-Stokes equations, (\ref{main.NS}), such that: 
\begin{align}
||u^\eps - \mu||_{\infty} + ||v^\eps||_\infty \le c_0(\mu)  \eps. 
\end{align}

The constant $c_0(\mu)$ satisfies: 
\begin{align}
c_0(\mu) \lesssim ||\frac{\mu'''}{\mu}||_{W^{100,\infty}}.
\end{align}
\end{theorem}

Our ultimate interest is motivated by Yudovich's ninth problem, \cite{Yudovich}. Classical experiments starting with Reynolds have shown that unsteady flows in a 2D channel that start near Couette or Poiseulle flow do not converge to these flows. This indicates the existence of infinitely many stationary solutions to Navier-Stokes ``near" Couette or Poiseulle. Establishing the existence of these solutions is an open problem. Our second result, Corollary \ref{Th.couette}, produces stationary solutions sufficiently close to Couette, assuming $x \in [0,L], L << 1$, and a moving boundary at $y = 2$. 

\begin{corollary} \label{Th.couette} Let any $\alpha > 0$ be prescribed, which could depend on $\eps$. Let $\tilde{\mu}$ be prescribed to satisfy the vanishing conditions: $\p_y^k \tilde{\mu}|_{y = 0} = \p_y^k \tilde{\mu}|_{y = 2} = 0$ for $0 \le k \le N_0$. There exists a unique solution, $\bold{u}^{\eps}$ to (\ref{main.NS}) with $u_b = 2$ such that: 
\begin{align}
||u^\eps - \Big( y + \alpha \tilde{\mu}(y) \Big) ||_{\infty} + ||v^\eps||_{\infty} \lesssim \alpha \eps.
\end{align}
\end{corollary}
\begin{proof}
One can obtain this by applying Theorem \ref{Th.no.slip} with $\mu(y) = y + \alpha \tilde{\mu}(y)$, where $\tilde{\mu}$ vanishes at high order near $y = 0, 2$. In this case, the constant $c_0(\mu) \lesssim \alpha$. 
\end{proof}

\begin{remark} The requirement of $u_b = 2$ is so that the no-slip condition is satisfied by the Couette flow. We do not use this motion of the boundary anywhere in the proof. 
\end{remark}

The present article is structured as follows: the construction of the approximate layers, $u_s, v_s$, in the expansion (\ref{asy.ex.1}) is performed in the Appendix. The main analysis in Sections \ref{sec.1}, \ref{sec.2} is centered around the system (\ref{sys.1}).

\vspace{3 mm}

\noindent \textbf{Acknowledgements:} The authors thank Yan Guo for many useful discussions regarding this problem. 

\section{Linear Estimates} \label{sec.1}

We will analyze the system (\ref{sys.1}). The reader is urged to consult Lemma \ref{Lemma.const.1} for relevant properties of the linearizations, $u_s$, and the forcing terms, $f, g$. 

\subsection{Energy Estimate}

\begin{proposition} For any $\theta > 0$, solutions $[u,v]$ to (\ref{sys.1}) satisfy: 
\begin{align} \label{en.est}
||u,v||_{\mathcal{E}}^2 + ||\sqrt{u_s} \{u, v\}||_{L^2(x = L)}^2 \lesssim C(\theta) \eps^{-\theta} ||u,v||_{\mathcal{P}}^2 + \mathcal{R}_1,
\end{align}

where:
\begin{align}
\mathcal{R}_1 := \int f \cdot u + \int \eps g \cdot v. 
\end{align}
\end{proposition}
\begin{proof}

Apply $[u,v]$ to (\ref{sys.1}). The coercive quantities are: 
\begin{align} \n
\int -\eps \Delta u \times u &- \int \eps \Delta v \times v + \int \nabla P \cdot \bold{u} \\ \n
& = \int \eps \Big[ - \p_{yy}u - 2 \p_{xx}u -  \p_{xy} v \Big] \times u + \int \p_x P u \\ \n
& \hspace{2 mm} + \int \eps \Big[ - 2\p_{yy}v - \p_x \{\p_y u + \p_x v \} \Big] \times v + \int \p_y P v \\ \n
& = \int \eps \Big[ |\p_y u^2 +  |\p_x v|^2 + 4 |\p_y v|^2 + 2\p_{x}v \p_y u \Big] \\
& \gtrsim \int \eps \Big[ |\nabla u|^2 + |\nabla v|^2 \Big].
\end{align}

Above, we have used the stress-free boundary condition in (\ref{bc.sf}). We now have the convection terms: 
\begin{align}
&\int [u_s u_x + u_{sx}u + v_s u_y] \cdot u= \int u_{sx} u^2 + \frac{1}{2} \int_{x = L} u_s u^2, \\
&\int [u_s v_x + v_s v_y + v_{sy}v ] \cdot v = \int v_{sy} v^2 + \frac{1}{2} \int_{x = L} u_s v^2.
\end{align}

We estimate the two bulk terms above using (\ref{prof.est.1}) - (\ref{prof.est.3}):
\begin{align}
|\int u_{sx}u^2| + |\int v_{sy} v^2| \le |\int \sqrt{\eps} \tilde{y} [u^2 + v^2]| \le \sqrt{\eps} \bigO(L) ||\sqrt{u_s} \nabla v||_2^2.
\end{align}

We now move to: 
\begin{align}
|\int v_{sx}uv| \le \eps ||\sqrt{\tilde{y}}u_x||_2 ||\sqrt{\tilde{y}}v_x||_2,
\end{align}

again by using (\ref{prof.est.3}). For the $u_{sy}v$ convection term, we first handle the leading order contribution from $\mu$, and we must take care to avoid the critical Hardy inequality: 
\begin{align} \n
|\int \mu' uv| &= |\int \mu' uv \Big[\chi(y \le \frac{1}{10}) + \chi(\frac{1}{10} \le y \le \frac{19}{10}) + \chi(y \ge \frac{19}{10})\Big] |.
\end{align}

For the interior contributions: 
\begin{align}
|\int \mu' uv \chi(\frac{1}{10} \le y \le \frac{19}{10})| \le \bigO(L) ||\sqrt{u_s} \p_x u||_2 ||\sqrt{u_s} \p_x v||_2 
\end{align}

The $y \le \frac{1}{10}$ contribution is exactly analogous to the $y \ge \frac{1}{10}$, and so we treat the former. Let $\tilde{\chi}$ denote a fattened relative to $\chi(y \le \frac{1}{10})$. Fix an $\omega > 0$ small.  
\begin{align} \label{fat.chi}
|\int \mu' uv \chi(y \le \frac{1}{10})| \le ||\mu'||_\infty  ||y^{-(\frac{1}{2}-\frac{\omega}{2})} v \tilde{\chi}||_2 || y^{\frac{1}{2}-\frac{\omega}{2}} u \tilde{\chi}||_2
\end{align}

We estimate each $L^2$ term above individually. 
\begin{align} \n
\int y^{-1+ \omega} v^2 \tilde{\chi} &= \int \frac{\p_y}{\omega} \{ y^{\omega} \} v^2 \tilde{\chi} \\ \n
& = - \int \frac{y^\omega}{\omega} 2v \p_y v \tilde{\chi} - \int \frac{y^{\omega}}{\omega} v^2 \tilde{\chi}' \\ \label{Ka.1}
& \le \frac{1}{\omega} ||y^{-(\frac{1}{2}-\frac{\omega}{2})} v \tilde{\chi}||_2 || \sqrt{u_s} \p_y v||_2 + \frac{1}{\omega} \bigO(L) ||\sqrt{u_s} \p_x v||_2^2. 
\end{align}

Next from (\ref{fat.chi}):
\begin{align} \n
||y^{\frac{1}{2}-\frac{\omega}{2}}u \tilde{\chi}||_2 &\lesssim ||y^{\frac{1}{2}}u \tilde{\chi}||_2^{1-\theta(\omega)} ||\frac{u}{y} \tilde{\chi}||_2^{\theta(\omega)} \\ \label{Ka.2}
& \lesssim \eps^{-\theta(\omega)} || \sqrt{u_s} \p_x u||_2^{1-\theta(\omega)}  \Big[ || \sqrt{\eps} \p_y u \tilde{\chi}||_2 + ||\sqrt{u_s} \p_x u \tilde{\chi}||_2 \Big]^{\theta(\omega)}
\end{align}

Inserting (\ref{Ka.1}) and (\ref{Ka.2}) into (\ref{fat.chi}), one obtains for small $\kappa > 0$
\begin{align}
|(\ref{fat.chi})| \le \kappa ||u,v||_{\mathcal{E}}^2 + N_{\kappa} \eps^{0-} ||u,v||_{\mathcal{P}}^2. 
\end{align}

For the higher-order contributions, we use the estimate (\ref{prof.est.2}), and subsequently split:
\begin{align} \label{conv.E}
|\int [u_{sy} - \mu'] uv| \le \int \sqrt{\eps}|u||v| [\chi_\delta^- + \chi_\delta^c + \chi_\delta^+] = (\ref{conv.E}.1) + (\ref{conv.E}.2) + (\ref{conv.E}.3). 
\end{align}

Here, $\chi_\delta^+(y) = \chi_\delta^-(2-y)$ and $\chi_\delta^c = 1 - \chi_\delta^+ - \chi_\delta^-$, where:
\begin{align}
\chi_\delta^-(y) = 
\left\{
\begin{aligned}
&1 \text{ on } y \le \delta \\
&0 \text{ on } y \ge 2\delta
\end{aligned}
\right. 
\end{align}

Terms (\ref{conv.E}.1) and (\ref{conv.E}.3) are identical.  We estimate:
\begin{align} \label{ins.0}
&|(\ref{conv.E}.1)| \lesssim \sqrt{\delta} \bigO(L) ||\sqrt{\eps}\frac{u}{y} \sqrt{\tilde{\chi}_\delta}||_2 ||\sqrt{u_s} u_x||_2, \\
&|(\ref{conv.E}.2)| \lesssim \frac{\sqrt{\eps}}{\delta} ||\sqrt{u_s} \nabla v||_2^2.
\end{align}

We now estimate:
\begin{align} \n
||\sqrt{\eps} \frac{u}{y} \sqrt{\tilde{\chi}_\delta}||_2^2 &= \int \eps \frac{\p_y}{-1} \{ y^{-1} \} u^2 \tilde{\chi}_\delta \\ \n
& = \int \eps y^{-1} 2u\p_y u \tilde{\chi}_\delta + \int \eps y^{-1} u^2 \frac{\tilde{\chi}_\delta'}{\delta} \\ \label{ins.1}
& \le  ||\sqrt{\eps} \frac{u}{y} \sqrt{\tilde{\chi}_\delta} ||_2 ||\sqrt{\eps} \p_y u||_2 + \frac{\eps}{\delta^3} ||\sqrt{u_s} \p_x u||_2^2. 
\end{align}

We may thus take $\delta = \eps^{\frac{1}{4}}$ and insert (\ref{ins.1}) into (\ref{ins.0}) to conclude. 

\end{proof}

\subsection{Positivity Estimate}
\begin{proposition} Solutions $[u,v]$ to (\ref{sys.1}) satisfy, for any $\kappa > 0$:
\begin{align} \label{pos.est}
||u,v||_{\mathcal{P}}^2 +||  \sqrt{\eps} \p_x v||_{L^2(x = 0)}^2 + ||\sqrt{\eps} \p_x u||_{L^2(x = L)}^2 \lesssim \eps^{1-\kappa} ||u,v||_{\E}^2 + \mathcal{R}_2,
\end{align}

where: 
\begin{align} \label{R2.defn.}
\mathcal{R}_2 := \int f \cdot - \p_y v + \int g \cdot \p_x v. 
\end{align}
\end{proposition}
\begin{proof}

We will apply the multiplier $\bold{M} := (-\p_y v, \p_x v)$ to the system (\ref{sys.1}). This gives: 
\begin{align} \n
\int \Big(-u_s \p_y v &+ v \p_y u_s \Big) \cdot -\p_y v + \int u_s \p_x v \cdot \p_x v \\ \n
& = \int  u_s \Big[ |\p_y v|^2 + |\p_x v|^2 \Big] + \int \frac{u_{syy}}{2}v^2 \\ \n
& \ge \int u_s |\nabla v|^2 - ||\frac{u_{syy}}{2}||_{\infty} \int v^2 \\
& \gtrsim \int u_s |\nabla v|^2. 
\end{align}

We have used the splitting: 
\begin{align} \n
\int v^2 &= \int v^2 [\chi_\delta + \chi_\delta^c] \\ \n
& \le \frac{L^2}{\delta} \int |\p_x v|^2 + |\int \p_y \{ y \} v^2 \chi_\delta | \\ \n
& \le \frac{L^2}{\delta} \int u_s |\p_x v|^2 + |\int y 2v \p_y v \chi_\delta | + |\int y v^2 \frac{\chi_\delta'}{\delta}| \\ \n
& \lesssim \frac{L^2}{\delta} \int u_s |\p_x v|^2 + \sqrt{\delta} || \sqrt{u_s} \p_y v||_2 \times \sqrt{\text{LHS}} \\ 
& \stackrel{\delta = L}{=} L \int u_s |\p_x v|^2 + \sqrt{L} || \sqrt{u_s} \p_y v||_2 \times \sqrt{\text{LHS}}.
\end{align}

For the vorticity terms, repeated integration by parts gives: 
\begin{align} \n
\int -\eps \Delta u \cdot \p_y v &- \int \eps \Delta v \cdot \p_x v + \int \nabla P \cdot \bold{M} \\ \n
& = \frac{\eps}{2} \int_{x = 0} |\p_x v|^2 + \frac{\eps}{2} \Big[ \int_{x = L} \Big( |\p_y u|^2 - |\p_x u|^2 + |\p_y v|^2 \\ \n &\hspace{10 mm} - |\p_x v|^2   \Big) \Big] + \int_{x = L} P \p_x u \\
& = \frac{\eps}{2} \int_{x = 0} |\p_x v|^2 + \int_{x = L} 2 \eps |\p_x u|^2,
\end{align}

where we have used the Stress-Free boundary condition from (\ref{bc.sf}). We now come to the remaining linearized terms from (\ref{sys.1}):
\begin{align} \n
|\int \Big( u \p_x u_s + v_s \p_y u \Big) &\cdot -\p_y v| + |\int \Big( u \p_x v_s + v_s \p_y v + v \p_y v_s \Big) \cdot \p_x v| \\ \n
&\lesssim \eps^\kappa \times \text{LHS of (\ref{pos.est})} + \eps^{1-\kappa} ||u,v||_{\mathcal{E}}^2,
\end{align}

where we have used the Poincare inequality and the estimates in (\ref{prof.est.1}) - (\ref{prof.est.3}). Finally, the right-hand side of (\ref{pos.est}) follows from the definition of $\mathcal{R}_2$.

\end{proof}

\begin{lemma} For any $\theta > 0$,
\begin{align} \label{unif.estimate}
\eps^{\theta} ||\sqrt{\eps}u, \sqrt{\eps}v||_\infty \le C_\theta \Big[ ||u,v||_{\E} + ||f,g||_{2} \Big].
\end{align}
\end{lemma}
\begin{proof}
We omit the proof, this is found in \cite{GN} using interpolation arguments and estimates for the Stokes operator on domains with corners. 
\end{proof}

As a direct corollary to (\ref{en.est}), (\ref{pos.est}), and taking $\theta = \frac{\gamma}{4}$ in (\ref{unif.estimate}):
\begin{corollary}
\begin{align}
||u,v||_X^2 \lesssim \mathcal{R}_1 + \mathcal{R}_2 + \eps^{\frac{\gamma}{2}} ||f,g||_{2}^2.
\end{align}
\end{corollary}

\section{Evaluation of Right-Hand Sides} \label{sec.2}

We first provide the nonlinear estimates:
\begin{lemma} With $\mathcal{N}_1, \mathcal{N}_2$ defined as in (\ref{defn.N}):
\begin{align} \n
|\int \eps^{\frac{3}{2}+\gamma} \mathcal{N}_1 \cdot [u + \p_x u]| &+ |\int \eps^{\frac{3}{2}+\gamma} \mathcal{N}_2 \cdot [v + \p_x v]| \\ &+ \eps^{\frac{\gamma}{4}}||\mathcal{N}_1, \mathcal{N}_2||_2 \le \eps^{\frac{\gamma}{2}} \Big[||u,v||_{X}^3 + ||u,v||_{X}^2\Big].
\end{align}
\end{lemma}
\begin{proof}

We compute directly:
\begin{align} \n
|\int \eps^{\frac{3}{2}+\gamma} [u\p_x u &+ v\p_y u]  \cdot [u + \p_x u]| \\ &\le \eps^{\frac{\gamma}{2}} ||\sqrt{\eps} \eps^{\frac{\gamma}{2}} \{u, v \}||_{\infty} ||\sqrt{\eps} \nabla \{u, v \}||_2^2 \lesssim \eps^{\frac{\gamma}{2}} ||u,v||_X^3. 
\end{align}

Similarly:
\begin{align} \n
|\int \eps^{\frac{3}{2}+\gamma} [u\p_x v &+ v\p_y v] \cdot [v + \p_x v]| \\ &\le \eps^{\frac{\gamma}{2}} ||\sqrt{\eps} \eps^{\frac{\gamma}{2}} \{u, v \}||_{\infty} ||\sqrt{\eps} \nabla \{u, v \}||_2^2 \lesssim \eps^{\frac{\gamma}{2}} ||u,v||_X^3. 
\end{align}

Finally: 
\begin{align}
||\mathcal{N}_1, \mathcal{N}_2||_2 \le \eps^{\frac{3}{2}+\gamma} \Big[ ||u \{ \p_x u, \p_x v \}||_2 + ||v \{ \p_y u, \p_y v \} ||_2 \Big]. 
\end{align}

\end{proof}

\begin{lemma} With $\mathcal{F}_u, \mathcal{F}_v$ defined as in (\ref{Fu.Fv}), for any $\delta > 0$:
\begin{align} \n
&\Big|\int \eps^{-\frac{3}{2}-\gamma} \mathcal{F}_{u} \cdot [u + \p_x u] + \int \eps^{-\frac{3}{2}-\gamma} \mathcal{F}_v \cdot [v + \p_x v] \Big| \\ \n
& \hspace{30 mm} \le \delta ||u,v||_{X}^2 + C(u_s, v_s)c_0(\frac{\mu'''}{\mu}), \\ \label{sd.1}
&\eps^{\frac{\gamma}{4}}||\eps^{-\frac{3}{2}-\gamma} \{ \mathcal{F}_u, \mathcal{F}_v \}||_2 \lesssim \eps^{\frac{1}{2}-\frac{3\gamma}{4}} c_0(\frac{\mu'''}{\mu}).
\end{align}
\end{lemma}
\begin{proof}
First recall the decomposition of $\mathcal{F}_u, \mathcal{F}_v$ given in (\ref{decomp.1}). We first estimate: 
\begin{align}
\int \eps^{\frac{1}{2}-\gamma} T_1 \cdot u = \int \eps^{\frac{1}{2}-\gamma} T_1 \cdot u [\chi_\delta + \chi_\delta^c].
\end{align}

For the nonlocal part, we use estimate (\ref{T1.1}): 
\begin{align}
|\int \eps^{\frac{1}{2}-\gamma} T_1 u \chi_\delta^c| \le \delta^{-\frac{1}{2}} \eps^{\frac{1}{2}-\gamma} ||T_1||_2 ||\sqrt{u_s} \p_x u||_2. 
\end{align}

For the local component, we integrate by parts in $y$:
\begin{align} \n
|\int \eps^{\frac{1}{2}-\gamma} T_1 u \chi_\delta(y)| &= \eps^{\frac{1}{2}-\gamma} |\int y \p_y u T_1 \chi_\delta + y u \p_y T_1 \chi_\delta + y u T_1 \frac{\chi_\delta'}{\delta}| \\ \n
& \le \delta \eps^{-\gamma} ||\sqrt{\eps} \p_y u||_2 ||T_1||_2 + \eps^{\frac{1}{2}-\gamma} \sqrt{\delta} ||\p_y T_1||_2 ||\sqrt{u_s}\p_x u||_2 \\ \label{pds.1} & + \delta^{-\frac{1}{2}} \eps^{\frac{1}{2}-\gamma} ||T_1||_2 ||\sqrt{u_s}\p_x u||_2. 
\end{align}

The same estimates can be used for $\eps^{\frac{1}{2}-\gamma} T_2 \cdot v$. We now come to the higher order terms, in which the non-local contributions are estimated via: 
\begin{align}
|\int \eps^{\frac{1}{2}-\gamma} T_1 \cdot \p_x u \chi_\delta^c  + \int \eps^{\frac{1}{2}-\gamma} T_2 \cdot \p_x v  \chi_\delta^c| \le \delta^{-\frac{1}{2}} \eps^{\frac{1}{2}-\gamma} ||T_1||_2 ||\sqrt{u_s} \nabla v||_2. 
\end{align}

We now focus on the $T_1$ localized contributions individually. First: 
\begin{align}
& \eps^{\frac{1}{2}-\gamma} |\int \mu' v^2_p \p_y v \chi_\delta| \le \eps^{-\gamma} ||\mu', \frac{v^2_p}{Y}||_\infty \sqrt{\delta} ||\sqrt{u_s} \p_y v||_2, \\
 &\eps^{\frac{1}{2}-\gamma} |\int \p_Y u^2_p v^1_e \p_y v \chi_\delta| \le \eps^{\frac{1}{2}-\gamma} ||\p_Y u^2_p||_{\infty} ||\frac{v^1_e}{y}||_2 ||\sqrt{u_s} \p_y v||_2, \\
 &\eps^{\frac{1}{2}-\gamma} |\int \Big[ \mu \chi' v^{2,0}_p + 3\chi' \p_Y u^{2,0}_p \Big] \p_y v| \le \eps^{\frac{1}{2}-\gamma} ||\mu \chi' v^{2,0}_p + 3\chi' \p_Y u^{2,0}_p||_{\infty} ||\sqrt{u_s} \p_y v||_2. 
\end{align}

The remaining terms in $T_1$ are handled by integrating by parts in $y$ and proceeding as in (\ref{pds.1}): 
\begin{align} \n
\eps^{\frac{1}{2}-\gamma} \int \Big[ u^1_e \p_x u^1_e &+ v^1_e \p_y u^1_e - \Delta u^1_e \Big] \cdot  \p_y v \chi_\delta \\ \n \
&= - \int \eps^{\frac{1}{2}-\gamma} \underbrace{\p_y \Big[ u^1_e \p_x u^1_e + v^1_e \p_y u^1_e - \Delta u^1_e \Big] }_{m_1} \cdot v \chi_\delta \\ \n & \hspace{10 mm} - \int \eps^{\frac{1}{2}-\gamma} \Big[ u^1_e \p_x u^1_e + v^1_e \p_y u^1_e - \Delta u^1_e \Big] \cdot v \frac{\chi_\delta'}{\delta} \\ \label{ml.1}
& = (\ref{ml.1}.1) + (\ref{ml.1}.2).
\end{align}

First: 
\begin{align}
|(\ref{ml.1}.1)| = |\int \eps^{\frac{1}{2}-\gamma} m_1 v| \le \eps^{\frac{1}{2}-\gamma} ||m_1||_2 ||v||_2 \le \eps^{\frac{1}{2}-\gamma} ||m_1||_2 ||\sqrt{u_s}\nabla v||_2.
\end{align}

Second:
\begin{align}
|(\ref{ml.1}.2)| \le \eps^{\frac{1}{2}-\gamma} \delta^{-\frac{3}{2}} ||u^1_e \p_x u^1_e + v^1_e \p_y u^1_e - \Delta u^1_e||_2 ||\sqrt{u_s} \p_x v||_2 
\end{align}

We now consider the localized contributions from $T_2$, for which we apply estimate (\ref{T1.1}):
\begin{align}
\eps^{\frac{1}{2}-\gamma}|\int T_2 \cdot  \p_x v| \le ||\frac{T_2}{\tilde{y}}||_{2} ||\sqrt{u_s} \p_x v||_2. 
\end{align}

We now make the selection of $\delta = \eps^{10 \gamma}$, and $\gamma << 1$ sufficiently small, which closes all of the above estimates. Finally, the $\bigO(\eps^{\frac{5}{2}})$ are handled easily via: 
\begin{align} \n
\eps^{-\frac{3}{2}-\gamma} &|\int [\mathcal{F}_u - T_1] \cdot [u + \p_x u] + \int [\mathcal{F}_v - T_2] \cdot [v + \p_x v]| \\ \n &\lesssim \int \eps^{1-\gamma} \cdot [u + \p_x u + v + \p_x v]| \\ 
& \lesssim \eps^{\frac{1}{2}-\gamma} ||\sqrt{\eps} \nabla v||_2.
\end{align}

\end{proof}

We now obtain our complete nonlinear estimate: 
\begin{corollary} Solutions $[u,v]$ to the system (\ref{sys.1}) satisfy: 
\begin{align}
||u,v||_X^2 \lesssim C(u_s, v_s) c_0(\frac{\mu'''}{\mu}) + \eps^{\frac{\gamma}{2}} ||u,v||_{X}^3.
\end{align}
\end{corollary}

From here, the main result, Theorem \ref{Th.no.slip} follows from a straightforward application of the contraction mapping theorem. 

%This is the AP Estimate, which we do not need to use here...

%This is the higher order (E_2) esitmate, which we do not need due to AP omitted...

%This is the Pressure esitmate, which we do not need due to AP omitted...

\iffalse %This is the Stokes and u_y|_0 esitmate, which we do not need due to AP omitted...
\subsection{Stokes Estimate}

\begin{lemma}
\begin{align}
||\eps (L-x)^\gamma \Big\{ \p_{yy}u, \p_{xy}u, \p_{xx} u, \p_{xx}v \Big\}||_{L^2} \lesssim C(L) \Big[ \Big]
\end{align}
\end{lemma}
\begin{proof}

\end{proof}

\subsection{$\p_y u$ Boundary Estimate}
\begin{lemma}

\end{lemma}
\begin{proof}

\end{proof}

\fi %This is the Stokes and u_y|_0 esitmate, which we do not need due to AP omitted...

\break

\appendix

\section{Construction of Layers}

We start with the asymptotic expansions:
\begin{align} \label{exp.u}
&u^\eps := \mu + \eps u^1_e + \eps u^1_p + \eps^{\frac{3}{2}} u^2_e + \eps^{\frac{3}{2}} u^2_p + \eps^{\frac{3}{2}+\gamma}u, \\ \label{exp.v}
&v^\eps := \eps v^1_e + \eps^{\frac{3}{2}} v^1_p + \eps^{\frac{3}{2}}v^2_e + \eps^2 v^2_p + \eps^{\frac{3}{2}+\gamma}v, \\ \label{exp.P}
&P^\eps := \eps P^1_e + \eps^{\frac{3}{2}}P^2_e + \eps[P^1_p + \eps P^{1,a}_p] + \eps^{\frac{3}{2}}P^2_p + \eps^{\frac{3}{2}+\gamma} P
\end{align}

\subsection{Formal Asymptotic Expansion}

Here the Eulerian profiles are functions of $(x,y)$, whereas the boundary layer profiles are functions of $(x,Y)$, where: 
\begin{align}
Y = 
\left\{
\begin{aligned}
&Y_+ := \frac{2-y}{\sqrt{\eps}} \text{ if } 1 \le y \le 2, \\
&Y_- := \frac{y}{\sqrt{\eps}} \text{ if } 0 \le y \le 1. 
\end{aligned}
\right.
\end{align}

Due to this, we break up the boundary layer profiles into two components, one supported near $y = 0$ and one supported near $y = 2$: 
\begin{align}
u^i_p(x,Y) = u^{i,-}_{p}(x,Y_-) + u^{i,+}_p(x,Y_+).  
\end{align}

As a \underline{notational convention}, we use: 
\begin{align}
\p_Y u^i_p := \p_{Y_-} u^{i,-}_p - \p_{Y_+} u^{i,+}_p. 
\end{align}

The purpose of such a convention is to obtain the chain rule: 
\begin{align}
\p_y u^i_p = \frac{1}{\sqrt{\eps}} \p_Y u^i_p. 
\end{align}

Let us set the following notations: 
\begin{align}
&u^\eps_E := \mu + \eps u^1_e + \eps^{\frac{3}{2}} u^2_e, \hspace{3 mm} v^\eps_E := \eps v^1_e + \eps^{\frac{3}{2}} v^2_e, \\
& u_s^{(2)} :=\mu + \eps u^1_e + \eps u^1_p + \eps^{\frac{3}{2}} u^2_e + \eps^{\frac{3}{2}} u^2_p, \\
& v_s^{(2)} :=  \eps v^1_e + \eps^{\frac{3}{2}} v^1_p + \eps^{\frac{3}{2}}v^2_e + \eps^2 v^2_p, \\
& P^{(2)}_s := P^\eps := \eps P^1_e + \eps^{\frac{3}{2}}P^2_e + \eps[P^1_p + \eps P^{1,a}_p] + \eps^{\frac{3}{2}}P^2_p.
\end{align}

Using the expansions (\ref{exp.u})  - (\ref{exp.P}), we will first expand out the purely Euler terms: 
\begin{align} \n 
u^\eps_E \p_x u^\eps_E &= \Big[ \mu + \eps u^1_e + \eps^{\frac{3}{2}}u^2_e \Big] \cdot \Big[ \eps u^1_{ex} + \eps^{\frac{3}{2}} u^2_{ex} \Big] \\ \label{up.1}
& = \eps \mu u^1_{ex} + \eps^2 u^1_e u^1_{ex} + \eps^{\frac{5}{2}} u^2_e u^1_{ex} + \eps^{\frac{3}{2}} \mu u^2_{ex} + \eps^{\frac{5}{2}} u^1_e u^2_{ex} + \eps^3 u^2_e u^2_{ex}
\end{align}

\begin{align} \n
v^\eps_E \p_y u^\eps_E &= \Big[\eps v^1_e + \eps^{\frac{3}{2}} v^2_e \Big] \cdot \Big[ \mu' + \eps u^1_{ey} + \eps^{\frac{3}{2}} u^2_{ey} \Big] \\ \label{up.2}
& = \eps \mu' v^1_e + \eps^2 v^1_e u^1_{ey} + \eps^{\frac{5}{2}} v^1_e u^2_{ey} + \eps^{\frac{3}{2}}\mu' v^2_e + \eps^{\frac{5}{2}} v^2_e u^1_{ey}+ \eps^3 v^2_e u^2_{ey}
\end{align}

\begin{align} \n
&u^\eps_E \p_x v^\eps_E = \Big[ \mu + \eps u^1_e + \eps^{\frac{3}{2}}u^2_e \Big]  \cdot \Big[ \eps v^1_{ex} + \eps^{\frac{3}{2}} v^2_{ex} \Big] \\ \label{up.3}
& \hspace{10 mm} = \eps \mu v^1_{ex}+ \mu \eps^{\frac{3}{2}} v^2_{ex} + \eps^2 u^1_e v^1_{ex} + \eps^{\frac{5}{2}} u^1_e v^2_{ex}+ \eps^{\frac{5}{2}} u^2_e v^1_{ex} + \eps^3 u^2_e v^2_{ex}, \\ \n
&v^\eps \p_y v^\eps_E = \Big[\eps v^1_e + \eps^{\frac{3}{2}} v^2_e \Big] \cdot \Big[ \eps v^1_{ey} + \eps^{\frac{3}{2}} v^2_{ey} \Big] \\ \label{up.4}
& \hspace{10 mm} = \eps^2 v^1_e v^1_{ey} + \eps^{\frac{5}{2}} v^1_e v^2_{ey} + \eps^{\frac{5}{2}} v^2_e v^1_{ey} + \eps^3 v^2_e v^2_{ey}.
\end{align}

\begin{align} \label{up.5}
&\p_x P^\eps_E = \eps P^1_{ex} + \eps^{\frac{3}{2}} P^2_{ex}, \\ \label{up.6}
&\p_y P^\eps_E = \eps P^1_{ey} + \eps^\frac{3}{2} P^2_{ey} \\ \label{up.7}
&\eps \Delta u^\eps_E = \eps \mu''(y) + \eps^2 \Delta u^1_e + \eps^{\frac{5}{2}} \Delta u^2_e, \\ \label{up.8}
&\eps \Delta v^\eps_E = \eps^2 \Delta v^1_e + \eps^{\frac{5}{2}} \Delta v^2_e.  
\end{align}

We now expand: 
\begin{align} \n
u_s^{(2)} \p_x u^{(2)}_s = &u^\eps_E \p_x u^\eps_E + \eps^2 u^1_p u^1_{ex} + \eps^2 u^1_p u^1_{px} + \eps^{\frac{5}{2}} u^1_p u^2_{ex} \\ \n
& + \eps \mu u^1_{px} + \eps^2 u^1_e u^1_{px} + \eps^{\frac{5}{2}} u^2_e u^1_{px} + \eps^{\frac{5}{2}} u^2_p u^1_{ex} + \eps^{\frac{5}{2}} u^2_p u^1_{px} \\ \n
& + \eps^3 u^2_p u^2_{ex} + \eps^3 u^2_p u^2_{px} + \eps^{\frac{3}{2}} \mu u^2_{px} + \eps^{\frac{5}{2}} u^1_e u^2_{px} \\ \label{all.1}
& + \eps^{\frac{5}{2}} u^1_p u^2_{px} + \eps^3 u^2_e u^2_{px}.
\end{align}

\begin{align} \n
v^{(2)}_s \p_y u^{(2)}_s =& v^\eps_E \p_y u^\eps_E + \eps^{\frac{3}{2}} \mu' v^1_p + \eps^{\frac{5}{2}} v^1_p u^1_{ey} + \eps^2 v^1_p u^1_{pY} + \eps^3 v^1_p u^2_{ey} \\ \n
& + \eps^{\frac{3}{2}} v^1_e u^1_{pY} + \eps^2 v^2_e u^1_{pY} + \eps^2 \mu' v^2_p + \eps^3 u^1_{ey} v^2_p + \eps^{\frac{5}{2}} v^2_p u^1_{pY} \\ \label{all.2}
& + \eps^{\frac{3}{2}} v^2_p u^2_{ey} + \eps^3 v^2_p u^2_{pY} + \eps^{\frac{5}{2}} v^2_e u^2_{pY} + \eps^{\frac{5}{2}} v^1_p u^2_{pY} + \eps^2 v^1_e u^2_{pY}.
\end{align}

\begin{align} \n
u^{(2)}_s \p_x v^{(2)}_s = & u^\eps_E \p_x v^\eps_E + \eps^2 u^1_p v^1_{ex} + \eps^{\frac{5}{2}} u^1_p v^1_{px} + \eps^{\frac{5}{2}} u^1_p v^2_{ex} + \eps^{\frac{3}{2}} \mu v^1_{px} \\ \n &+ \eps^{\frac{5}{2}} u^1_e v^1_{px} + \eps^3 u^2_e v^1_{px} + \eps^2 \mu v^2_{px} + \eps^3 u^1_e v^2_{px} \\ \n
& + \eps^3 u^1_p v^2_{px} + \eps^{\frac{5}{2}} u^2_e v^2_{px} + \eps^{\frac{7}{2}} u^2_p v^2_{px} + \eps^3 u^2_p v^2_{ex} \\ \label{all.3}
& + \eps^3 u^2_p v^1_{px} + \eps^{\frac{5}{2}} v^1_{ex}u^2_p.
\end{align}

\begin{align} \n
v^{(2)}_s \p_y v^{(2)}_s = &v^\eps_E \p_y v^\eps_E + \eps^{\frac{5}{2}} v^1_p v^1_{ey} + \eps^{\frac{5}{2}} v^1_p v^1_{pY} + \eps^3 v^1_p v^2_{ey} + \eps^2 v^1_e v^1_{pY} \\ \n
& + \eps^{\frac{5}{2}} v^2_e v^1_{pY} + \eps^3 v^1_{ey} v^2_p + \eps^3 v^1_{pY} v^2_p + \eps^{\frac{7}{2}} v^2_p v^2_{ey} + \eps^{\frac{7}{2}} v^2_p v^2_{pY} \\ \label{all.4} &+ \eps^{\frac{5}{2}} v^1_e v^2_{pY} + \eps^3 v^1_p v^2_{pY} + \eps^3 v^2_e v^2_{pY}.
\end{align}

Finally, we have the linear terms: 
\begin{align} \label{all.5}
&\p_x P_s = \p_x P^\eps_E + \eps^{\frac{3}{2}} P^2_{px} + \eps P^1_{px} + \eps^2 P^{1,a}_{px}, \\  \label{all.6}
& \p_y P_s = \p_y P^\eps_E + \eps P^2_{pY} + \sqrt{\eps} P^1_{pY} + \eps^{\frac{3}{2}} P^{1,a}_{pY} \\\label{all.7}
&\eps \Delta u^\eps = \eps \Delta u^\eps_E + \eps^2 u^1_{pxx} + \eps^{\frac{5}{2}} u^2_{pxx} + \eps u^1_{pYY} + \eps^{\frac{3}{2}} u^2_{pYY}, \\ \label{all.8}
& \eps \Delta v^\eps = \eps \Delta v^\eps_E + \eps^{\frac{5}{2}} v^1_{pxx} + \eps^{\frac{3}{2}} v^1_{pYY} + \eps^3 v^2_{pxx} + \eps^2 v^2_{pYY}.
\end{align}

\subsection{Euler Equations}

The equations satisfied by the Euler layers are obtained by collecting the $\bigO(\eps)$ order terms from (\ref{up.1}) - (\ref{up.8}), and is now shown: 
\begin{align} \label{v1e.BVP}
\left.
\begin{aligned}
& \mu \p_x u^1_e + \mu' v^1_e + \p_x  P^1_e = \mu''(y) \\
& \mu \p_x v^1_e + \p_y P^1_e = 0, \\
& \p_x u^1_e + \p_y v^1_e = 0, \\
& v^1_e|_{x = 0} = v^1_e|_{y = 0} = v^1_e|_{y = 2} = v^1_e|_{x = L} = 0. 
\end{aligned}
\right\}
\end{align}

By going to the vorticity formulation, we arrive at the following problem: 
\begin{align} \label{eqn.eul.vort.1}
-\mu \Delta v^1_e + \mu'' v^1_e = \mu'''(y), \hspace{3 mm} v^1_e|_{\p \Omega} = 0, \hspace{3 mm} u^1_e := \int_0^x v^1_{ey}.
\end{align}

We will make the assumptions that: 
\begin{align} \label{as.euler.1}
\frac{\mu''}{\mu}, \frac{\mu'''}{\mu} \text{ vanish at high order at $y = 0, 2$.}
\end{align}

According to (\ref{as.euler.1}), we divide (\ref{eqn.eul.vort.1}) by $\mu$ to obtain: 
\begin{align} \label{divide.1}
-\Delta v^1_e + \frac{\mu''}{\mu} v^1_e = \frac{\mu'''}{\mu}, \hspace{5 mm} v^1_e|_{\p \Omega} = 0. 
\end{align}

By evaluating (\ref{divide.1}) at $y = 0, 2$ and recalling (\ref{as.euler.1}), it is clear that $\p_{yy} v^1_e|_{y = 0, 2} = 0$. The system satisfied by the second Euler layer is obtained by collecting the $\bigO(\eps^{\frac{3}{2}})$ terms from (\ref{up.1}) - (\ref{up.8}), and is shown here: 
\begin{align} \label{v2e.BVP}
\left.
\begin{aligned}
& \mu \p_x u^2_e + \mu' v^2_e + \p_x  P^2_e = 0 \\
& \mu \p_x v^2_e + \p_y P^2_e = 0, \\
& \p_x u^2_e + \p_y v^2_e = 0, \\
& v^2_e|_{x = 0} = v^2_e|_{x = L} = v^2_e|_{y = 2} = 0, \hspace{3 mm} v^2_e|_{y = 0} = -v^1_p|_{Y = 0}.
\end{aligned}
\right\}
\end{align}

Going to vorticity produces the system: 
\begin{align}
-\mu \Delta v^2_e + \mu'' v^2_e = 0, \hspace{3 mm} v^2_e|_{y = 0, 2} = - v^1_p|_{y = 0, 2}.
\end{align}

We will assume high-order compatibility conditions on the data $v^2_e|_{x = 0, L}$ with $v^2_e|_{y = 0,2}$ at the four corners of the domain, $\Omega$. The first of these conditions at the corner $x = 0, y = 0$ is as follows:
\begin{align} \label{compat.1}
\p_{yy}v^1_e|_{x= 0}(0) = \p_{yy}v^1_e|_{y = 0}(0) = -\frac{\mu''}{\mu} v^1_p|_{y = 0}. 
\end{align}

The remaining compatibility conditions may be derived in the same manner. These will contribute higher order terms, which are the $\bigO(\eps^2)$ terms from (\ref{up.1}) - (\ref{up.8}): 
\begin{align} \n
\mathcal{C}_{1,u} := &\eps^2[u^1_e \p_x u^1_e + \sqrt{\eps}u^2_e \p_x u^1_e + \sqrt{\eps} u^1_e \p_x u^2_e + \eps u^2_e \p_x u^2_e] + \\ \n
& \eps^2 \Big[ v^1_e \p_y u^1_e + \sqrt{\eps}v^2_e \p_y u^1_e + \sqrt{\eps}v^1_e \p_y u^2_e + \eps v^2_e \p_y u^2_e \Big] \\
& -\eps^2 \Delta u^1_e - \eps^{\frac{5}{2}} \Delta u^2_e, \\ \n
\mathcal{C}_{1,v} := &\eps^2 u^1_e \p_x v^1_e + \eps^{\frac{5}{2}} u^1_e \p_x v^2_e + \eps^{\frac{5}{2}} u^2_e \p_x v^1_e + \eps^3 u^2_e \p_x v^2_e \\ \n
& + \eps^2 v^1_e \p_y v^1_e + \eps^{\frac{5}{2}} v^2_e \p_y v^1_e + \eps^{\frac{5}{2}} v^1_e \p_y v^2_e + \eps^3 v^2_e \p_y v^2_e \\
& - \eps^2 \Delta v^1_e - \eps^{\frac{5}{2}} \Delta v^2_e. 
\end{align}

The following follow from standard elliptic theory: 
\begin{lemma}
Assuming (\ref{as.euler.1}) and compatibility conditions for both $v^1_e, v^2_e$ for arbitrary order as in (\ref{compat.1}), there exist unique solutions, $v^1_e, v^2_e$ to (\ref{v1e.BVP}) and (\ref{v2e.BVP}) that are regular:  
\begin{align}
| \p_x^l \p_y^m \{ u^i_e, v^i_e \}| \lesssim c_0(\frac{\mu'''}{\mu}) \times C_{l,k} \text{ for } i = 1,2.
\end{align}
\end{lemma}

\subsection{Boundary Layer Equations}

Collecting the $\bigO(\eps)$ terms from (\ref{all.1})- (\ref{all.8}):
\begin{align} \label{pr.BVP.1}
\left.
\begin{aligned}
&\mu \p_x u^{1,0, -}_p - \p_{Y_-Y_-}u^{1,0, -}_p = 0, \hspace{3 mm} \p_{Y_-} P^{1,0, -}_p = 0, \\
&u^{1,0, -}_p|_{x = 0} = , \hspace{3 mm} u^{1,0, -}_p|_{Y_- = 0} = -u^1_e|_{y = 0}, u^{1,0, -}_p|_{Y_- \rightarrow \infty} = 0 \\
&v^{1,0, -}_p = \int_{Y_-}^\infty \p_x u^{1,0, -}_p. 
\end{aligned}
\right\}
\end{align}

Here we must assume the compatibility condition: 
\begin{align} \label{comp.BL}
u^{1,0, -}_p(0,Y_-)|_{Y_- = 0} = -u^1_e|_{y = 0}, \hspace{3 mm} \p_{Y_-}^2 u^{1,0, -}_p(0,Y)|_{Y_- = 0} = 0. 
\end{align}

We will also assume higher order compatibility conditions that can be obtained by differentiating the above system and reading the resulting equalities. Note that we construct $u^{1,0, -}_p, v^{1,0, -}_p$ on $(0,L) \times (0,\infty)$. We now cut-off these layers and make a $\bigO(\sqrt{\eps})$-order error: 
\begin{align} \label{cut.off.1}
u^{1,-}_p = \chi(\frac{\sqrt{\eps}Y}{100}) u^{1,0, -}_p - \frac{\sqrt{\eps}}{100} \chi'(\frac{\sqrt{\eps}Y}{100}) \int_0^x v^{1,0,-}, \hspace{2 mm} v^{1,-}_p := \chi(\frac{\sqrt{\eps}Y}{100}) v^{1,0, -}_p
\end{align}

$u^{1, 0, +}, v^{1,0,+}, u^{1,+}, v^{1,+}$ are defined analogously, and we omit these details. We then define: 
\begin{align}
&u^{1,0}_p := u^{1,0,-}_p + u^{1,0,+}_p, \hspace{5 mm} v^1_p := v^{1,0,-}_p + v^{1,0,+}_p, \\
&u^1_p := u^{1,-}_p + u^{1,+}_p, \hspace{5 mm} v^1_p := v^{1,-}_p + v^{1,+}_p.
\end{align}

Note that due to the cut-off in (\ref{cut.off.1}), $u^1_p, v^1_p$ is smooth. The contributions to the next layer are: 
\begin{align} \n
\mathcal{C}_{2,u} := &\eps^2 \p_x P^{1,a}_p + \eps^2 u^1_e \p_x u^1_p + \eps^2 u^1_p \p_x u^1_e + \eps^2 u^1_p\p_x u^1_p \\ \n
& + \eps^{\frac{3}{2}} v^1_e \p_Y u^1_p + \eps^{\frac{3}{2}} v^1_p \mu' + \eps^{\frac{5}{2}} v^1_p \p_y u^1_e + \eps^2 v^1_p \p_Y u^1_p \\
& - \eps^2 \p_{xx}u^1_p + \Big[ \eps^{\frac{5}{2}} \Big( u^2_e u^1_{px} + u^1_p u^2_{ex} \Big) + \eps^2 v^2_e u^1_{pY} + \eps^3 v^1_p u^2_{ey} \Big] + \mathcal{C}^1_{cut}.
\end{align}

Here $\mathcal{C}^1_{cut}$ is the error introduced by the cut-off functions in (\ref{cut.off.1}):
\begin{align} \n
\mathcal{C}^{1}_{cut} := &\sqrt{\eps} \mu \chi' v^{1,0}_p + 3\sqrt{\eps} \chi' \p_Y u^{2,0}_p \\ &+ 3\eps \chi'' u^{2,0}_p - \eps^{\frac{3}{2}} \chi''' \int_Y^\infty u^{2,0}_p, 
\end{align}

Define the auxiliary pressure via: 
\begin{align} \n
P^{1,a}_p := -\int_Y^1 \Big[ &\mu \p_x v^1_p + \eps u^1_e \p_x v^1_p + \sqrt{\eps} u^1_p \p_x v^1_e + \eps u^1_p \p_x v^1_p \\ \n
& + \sqrt{\eps} v^1_e \p_Y v^1_p + \eps v^1_p \p_y v^1_e + \eps v^1_p \p_Y v^1_p - \p_{YY}v^1_p - \eps \p_{xx}v^1_p \\ 
& + \eps^{-\frac{3}{2}}\Big( \eps^3 u^2_e v^1_{px} + \eps^{\frac{5}{2}} u^1_p v^2_{ex} + \eps^{\frac{5}{2}} v^2_e v^1_{py} + \eps^3 v^1_p v^2_{ey} \Big) \Big].
\end{align}

With such a choice, 
\begin{align}
\mathcal{C}_{2,v} := 0.
\end{align}

Collecting the $\bigO(\eps^{\frac{3}{2}})$ terms from (\ref{all.1}) - (\ref{all.8}), the system satisfied by the second boundary layers is:
\begin{align} \label{pr.BVP.2}
\left.
\begin{aligned}
&\mu \p_x u^{2,0}_p - \p_{YY}u^{2,0}_p = f_2 := \eps^{-\frac{3}{2}}\mathcal{C}_{2,u}, \hspace{3 mm} \p_Y P^2_p = 0, \hspace{3 mm} \p_x u^{2,0}_p + \p_Y v^{2,0}_p = 0, \\
& u^{2,0}_p|_{x = 0} = , \hspace{3 mm} u^{2,0, -}_p|_{Y = 0} = -u^2_e|_{y = 0}, \hspace{3 mm} u^{2,0,+}_p|_{y = 2} = -u^2_e|_{y = 2}, \\
& u^{2,0, -}_p|_{Y \rightarrow \infty} = 0, \hspace{3 mm} u^{2,0,+}_p|_{Y \rightarrow -\infty} = 0 \\
& v^{2,0, -}_p = - \int_0^{Y_-} \p_x u^{2,0,-}_p, \hspace{3 mm} v^{2,0,+}_p = -\int_2^{Y_+} \p_x u^{2,0,+}_p. 
\end{aligned}
\right\}
\end{align}

Note that in the same manner as in $u^{1}_p, v^1_p$, we have two boundary layer variables, $Y_-, Y_+$. We compactify the notation in (\ref{pr.BVP.2}) to simultaneously address both. Define the cut-off layer via:
\begin{align}
u^{2}_p := \chi(\frac{\sqrt{\eps}Y}{100}) u^{2,0}_p - \frac{\sqrt{\eps}}{100} \chi'(\frac{\sqrt{\eps}Y}{100}) \int_0^x v^2_p, \hspace{3 mm} v^{2}_p := \chi(\frac{\sqrt{\eps}Y}{100}) v^{2,0}_p.
\end{align}

\begin{lemma} Assume high order compatibility conditions in the sense of (\ref{comp.BL}) for both $u^1_p, u^2_p$. There exist unique solutions to (\ref{pr.BVP.1}) and (\ref{pr.BVP.2}) that are regular and satisfy the following estimates: 
\begin{align}
&|Y^m \p_x^k \p_Y^l \{u^1_p, v^1_p\}| \le c_0(\frac{\mu'''}{\mu}) \times C_{m,k,l} \text{ for any } k,l,m \ge 0, \\
&|Y^m \p_x^k \p_Y^l u^2_p| \le c_0(\frac{\mu'''}{\mu}) \times C_{m,k,l} \text{ for any } k, l, m \ge 0, \\
&|\p_x^k v^2_p| \le c_0(\frac{\mu'''}{\mu})\times C_{k} \text{ for any } k \ge 0. 
\end{align}
\end{lemma}
\begin{proof}
These follow from standard heat equation estimates. 
\end{proof}

The following are the errors contributed to the next layer: 
\begin{align} \n
\mathcal{C}_{3,u} := &\eps^{\frac{5}{2}}[u^1_e + u^1_p + \sqrt{\eps}u^2_p + \sqrt{\eps}u^2_e] \p_x u^2_p + \eps^{\frac{3}{2}} u^2_p[\eps \p_x u^1_e + \eps \p_x u^1_p \\ \n 
&+ \eps^{\frac{3}{2}} \p_x u^2_p + \eps^{\frac{3}{2}} \p_x u^2_e] + \eps \p_Y u^2_p \Big[ \eps v^1_e + \eps^{\frac{3}{2}} v^1_p + \eps^{\frac{3}{2}} v^2_e + \eps^2 v^2_p \Big] \\ 
& + \eps^2 v^2_p \Big[ \mu' + \eps \p_y u^1_e + \sqrt{\eps} \p_Y u^1_p + \eps^{\frac{3}{2}} \p_y u^2_e  \Big] - \eps^{\frac{5}{2}} \p_{xx}u^2_p + \eps^{\frac{3}{2}} \mathcal{C}_{cut}, \\ \n
\mathcal{C}_{3,v} := &\eps^2 \p_x v^2_p \Big[ \mu + \eps u^1_e + \eps u^1_p + \eps^{\frac{3}{2}} u^2_e \Big] + \eps^{\frac{3}{2}} u^2_p \Big[ \eps \p_x v^1_e + \eps^{\frac{3}{2}} \p_x v^1_p + \eps^{\frac{3}{2}} \p_x v^2_e \Big] \\ \n &+ \eps^{\frac{7}{2}} u^2_p \p_x v^2_p + \eps^2 v^2_p \Big[ \eps \p_y v^1_e + \eps \p_y v^1_p + \eps^{\frac{3}{2}} \p_y v^2_e \Big] \\ 
& + \eps^{\frac{3}{2}} \p_Y v^2_p \Big[ \eps v^1_e + \eps^{\frac{3}{2}} v^1_p + \eps^{\frac{3}{2}} v^2_e \Big] + \eps^{\frac{7}{2}} v^2_p \p_y v^2_p.
\end{align}

Here $\mathcal{C}_{cut}$ is the error contributed by cutting off the layers: 
\begin{align} \n
\mathcal{C}_{cut} := &(1- \chi) f_2 + \sqrt{\eps} \mu \chi' v^{2,0}_p + 3\sqrt{\eps} \chi' \p_Y u^{2,0}_p \\ &+ 3\eps \chi'' u^{2,0}_p - \eps^{\frac{3}{2}} \chi''' \int_Y^\infty u^{2,0}_p. 
\end{align}

The total contributions to the remainder forcing is: 
\begin{align} \label{Fu.Fv}
\mathcal{F}_u := \mathcal{C}_{1,u} + \mathcal{C}_{3,u}, \hspace{3 mm} \mathcal{F}_v := \mathcal{C}_{1,v} + \mathcal{C}_{3,v}.
\end{align}

We can break up the forcing contribution into: 
\begin{align} \label{decomp.1}
\mathcal{F}_u = \mathcal{T}_{u,\eps^2} + \bigO(\eps^{\frac{5}{2}}), \hspace{5 mm} \mathcal{F}_v = \mathcal{T}_{v,\eps^2} + \bigO(\eps^{\frac{5}{2}}),
\end{align}

where the terms at $\bigO(\eps^2)$ are the following: 
\begin{align}
&\mathcal{T}_{u, \eps^2} := \eps^2 \underbrace{\Big[ u^1_e \p_x u^1_e + v^1_e \p_y u^1_e - \Delta u^1_e + \p_Y u^2_p v^1_e + \mu' v^2_p + \mu \chi' v^{2,0}_p + 3 \chi' \p_Y u^{2,0}_p \Big]}_{T_1}  \\
&\mathcal{T}_{v, \eps^2} := \eps^2 \underbrace{\Big[ u^1_e \p_x v^1_e + v^1_e \p_y v^1_e - \Delta v^1_e + \mu \p_x v^2_p \Big]}_{T_2}. 
\end{align}

Summarizing the above constructions: 
\begin{lemma} \label{Lemma.const.1} The following estimates are satisfied by $u_s, v_s$:
\begin{align} \label{prof.est.1}
&|\p_x u_s|+ |\p_y v_s| + |u_s - \mu| \le \min \{ \bigO(\sqrt{\eps})\tilde{y}, \bigO(\eps) \}, \\ \label{prof.est.2}
&|\p_y u_s - \mu'| \lesssim \sqrt{\eps}, \\ \label{prof.est.3}
&|\p_x^l v_s| \lesssim \eps \tilde{y} \text{ for } l \ge 0,
\end{align}

The following are satisfied by $T_{u, \eps^2}, T_{v,\eps^2}$:
\begin{align} \label{T1.1}
&||T_1, \frac{T_2}{\tilde{y}}||_2  \lesssim c_0(\frac{\mu'''}{\mu}), \\ \label{T1.2}
&||\p_y T_1, \p_y T_2||_{2}  \lesssim c_0(\frac{\mu'''}{\mu}) \times \eps^{-\frac{1}{4}}.
\end{align}
\end{lemma}
\begin{proof}
Only $\frac{T_2}{y}$ is non-trivial, and it follows by examining that all terms in $T_2$ satisfy $T_2|_{y = 0} = 0 = T_2|_{y = 2}$. Note that we have used (\ref{eqn.eul.vort.1}) and (\ref{as.euler.1}) to conclude that $\Delta v^1_e|_{y = 0, 2} = 0$. 
\end{proof}

\break

%\iffalse

\renewcommand\refname{References Cited}
\pagenumbering{gobble}
\def\bibindent{3.5em}

\end{document}